\documentclass[11pt]{article}
\usepackage{anysize}\marginsize{3.5cm}{3.5cm}{1.3cm}{2cm}
\makeatletter\@ifundefined{pdfpagewidth}{}{\pdfpagewidth=21.0cm\pdfpageheight=29.7cm}\makeatother
\usepackage{amssymb,amsmath,graphicx,array}

\makeatletter
\let\orig@item=\@item \def\@item[#1]{\orig@item[\rm #1]}
\renewenvironment{abstract}{\begin{quote}\footnotesize\textbf{\abstractname.}}{\end{quote}\bigskip}
\renewcommand\@seccntformat[1]{\csname the#1\endcsname.\enspace}
\renewcommand\paragraph{\@startsection{paragraph}{4}{\z@}{1\baselineskip}{-0.5em}{\normalsize\bfseries}}
\newcommand\secparagraph{\@startsection{subsection}{2}{\z@}{1\baselineskip}{-0.5em}{\normalsize\bfseries}}
   
\let\origcaption=\caption \renewcommand\caption[1]{\parbox{0.66\textwidth}{\origcaption{#1}}}

\renewcommand\@begintheorem[2]{\trivlist\item[\hskip\labelsep{\bfseries#1 #2.}]\it}
\renewcommand\@opargbegintheorem[3]{\trivlist\item[\hskip\labelsep{\bfseries#1 #2}] {\bfseries(#3).}\enspace\it\ignorespaces}
\makeatother

\newtheorem{satz}{Satz}[section]
\makeatletter\@addtoreset{equation}{satz}\makeatother

\newtheorem{proposition}[satz]{Proposition}

\newtheorem{example}[satz]{Example}

\newenvironment{definition}[1][\bf Definition]{\trivlist\item[\hskip\labelsep{\it #1.}]}{\bigskip\\}
\newenvironment{acknowledgements}[1][\it Acknowledgements]{\trivlist\item[\hskip\labelsep{\it #1.}]}{\bigskip\\}
\newenvironment{remark}[1][\bf Remark]{\trivlist\item[\hskip\labelsep{\it #1.}]}{\bigskip\\}

\newtheorem{introtheorem}{Theorem}

\newenvironment{proof}[1][Proof]{\trivlist\item[\hskip\labelsep{\it #1.}]}{\hspace*{\fill}$\Box$\endtrivlist}

\newcommand\grant[1]{{\renewcommand\thefootnote{}\footnotetext{#1.}}}
\newcommand\subjclass[1]{{\renewcommand\thefootnote{}\footnotetext{2010 \textit{Mathematics Subject Classification:} #1.}}}
\newcommand\keywords[1]{{\renewcommand\thefootnote{}\footnotetext{\textit{Keywords.} #1.}}}
\newcommand\engqq[1]{``#1''}

\renewcommand\ge{\geqslant}  
\renewcommand\le{\leqslant}  
\renewcommand\epsilon{\varepsilon}
\renewcommand\phi{\varphi}

\renewcommand\subset{\subseteq}
\renewcommand\supset{\supseteq}
\renewcommand\tilde{\widetilde}

\renewcommand\P{\mathbb P}

\newcommand\be{\begingroup\arraycolsep=0.13888em\begin{eqnarray*}}
\newcommand\ee{\end{eqnarray*}\endgroup}

\newcommand\set[1]{\left\{#1\right\}}

\newcommand\with{\ \vrule\ }

\newlength\matrcolsep \matrcolsep=\arraycolsep

\newcommand\tline{\noalign{\vskip0.4ex}\hline\noalign{\vskip0.65ex}}

\newcommand\newop[2]{\newcommand#1{\mathop{\rm #2}\nolimits}}

\newcommand\Q{\mathbb Q}
\newcommand\R{\mathbb R}
\newcommand\Z{\mathbb Z}

\newcommand\PtP{\P^1\times\P^1}

\newop{\dv}{div}
\newop{\id}{id}
\newop{\ord}{ord}
\newop{\Div}{Div}
\newop{\Face}{Face}
\newop{\Vol}{Vol}
\newop{\Neg}{Neg}
\newop{\Cox}{Cox}
\newop{\Null}{Null}
\newop{\Nef}{Nef}
\newop{\supp}{supp}
\newop{\NE}{\overline{{NE}}}
\newop{\Eff}{\overline{{Eff}}}
\newop{\valnu}{\nu_{Y_\bullet}}
\newop\NS{NS} 
\newop\vol{vol}
\newop\BigCone{Big}

\begin{document}

\title{Minkowski decomposition of Okounkov bodies on surfaces}
\author{Patrycja \L{}uszcz-\'Swidecka and David Schmitz}
\date{}
\maketitle
\thispagestyle{empty}

\grant{The second author was supported by DFG grant BA 1559/6-1}

\begin{abstract}
	We prove that the Okounkov body of a big divisor with respect to 
	a general flag on a smooth projective surface whose pseudo-effective cone is rational 
	polyhedral decomposes as the Minkowski sum of finitely many simplices and line segments
	arising as Okounkov bodies of nef divisors.
\end{abstract}


\section{Introduction}
The construction of Okounkov bodies associated to linear series 
on a projective variety, which was introduced by Okounkov 
and was given a theoretical framework in the seminal papers 
\cite{KK} 
and \cite{LM}, recently attracted 
attention as it encodes plenty of information on geometric 
properties of line bundles. For example, the volume of a big 
linear series essentially agrees with the euclidean volume of 
its associated Okounkov body.

Okounkov's idea is to assign to a big divisor $D$ on a smooth projective
$n$-dimensional variety $X$
a convex body $\Delta(D)$ in $n$-dimensional euclidean space $\R^n$. 
The construction, which we sketch in section \ref{surfaces}, depends on the 
choice of a flag of subvarieties  
$Y_\bullet: X=Y_0 \supset Y_1\supset\dots\supset Y_n$ of codimensions $i$ 
such that $Y_n$ is a non-singular point on each of the $Y_i$. 

In (\cite[Theorem B]{LM}), Lazarsfeld and Musta\c{t}\v{a}  prove the existence 
of a \emph{global Okounkov body}: for a smooth projective variety there is a 
closed convex cone $\Delta(X)\subset \R^n\times N^1(X)_\R$ such that the fiber 
over any big rational class $\xi\in N^1(X)_\R$ of the map $\phi$ induced by 
the second projection is equal to $\Delta(\xi)$. 
Additionally, in order to establish the log-concavity relation
$$
	\vol_X(D_1+D_2)^{1/n} \ge \vol_X(D_1)^{1/n}+\vol_X(D_2)^{1/n}
$$
for any two big $\R$-divisors, they deduce from the convexity of 
the global Okounkov body the inclusion
$$
	\Delta(D_1)+\Delta(D_2)\subset\Delta(D_1+D_2).
$$	
Here the left hand side denotes the Minkowski sum of $\Delta(D_1)$ and 
$\Delta(D_2)$, i.e., the set obtained by pointwise addition 
see \cite[Corollary 4.12]{LM}). 

In general the above inclusion turns out to be strict 
(see Example \ref{ex:not unique}). However, it would be 
desirable to know conditions for equality; in particular one would hope
to be able to decompose the Okounkov body of any big divisor 
as the Minkowski sum of  \engqq{simple} bodies. Specifically, 
the following questions arise: is there a set $\Omega$ 
of big divisors such that the Okounkov body of  any big divisor $D$ 
with respect to an admissible flag $Y_\bullet$
decomposes as Minkowski sum of the bodies associated to divisors 
in $\Omega$? If so, can $\Omega$ be chosen to be finite? 

An affirmative answer to these questions was given in \cite{L-S} in the case 
of the del Pezzo surface $X_3$, the blow-up of the projective plane 
in three non-collinear points, equipped with a certain natural flag. 
We prove in this paper that the answers to both questions are \engqq{yes} 
for a general admissible flag (see Proposition \ref{prop:general flag})
 on any smooth projective surface 
whose pseudo-effective cone is rational polyhedral. For example, 
this is the case for all del Pezzo surfaces and,
more generally, for surfaces with big anticanonical class (see 
\cite[Lemma 3.4]{BS}).
We will see in the following section that considering nef divisors is 
sufficient since the Okounkov body of any big divisor is a translate of the 
body associated to the positive part of its Zariski decomposition.  
\begin{introtheorem}
	Let $X$ be a smooth projective surface such that $\Eff(X)$ is 
	rational polyhedral, and let $X=Y_0\supset Y_1 \supset Y_2=\set{pt}$
	be a general flag. Then there exists a finite set $\Omega$ of nef $\Q$-
	divisors such that for any nef $\Q$-divisor $D$ there exist non-negative 
	rational numbers $\alpha_P(D)$ such that
	\begin{equation}\label{eq:theorem}
		D=\sum_{P\in\Omega}\alpha_P(D) P \qquad \text{and}\qquad 
		\Delta_{Y_\bullet}(D)=\sum_{P\in\Omega}\alpha_P(D)\Delta_{Y_\bullet}
			(P).
	\end{equation}
\end{introtheorem}
\begin{definition}
	A presentation $D=\sum \alpha_i D_i$ as in (\ref{eq:theorem}) is called a 
	\emph{Minkowski decomposition} of $D$ with respect to 
	the \emph{Minkowski basis} $\Omega$.
\end{definition}
The proof, which we present in section \ref{proof}, includes the construction
of the Minkowski basis $\Omega$ as well as an effective method to 
determine a Minkowski decomposition of any given nef $\Q$-divisor. 
It depends on two features distinctive for surfaces, firstly a 
characterization of Okounkov bodies in terms of intersections with the 
positive and negative part in the Zariski decomposition due to Lazarsfeld and 
Musta\c{t}\v{a}, and secondly on the Zariski chamber decomposition of the big 
cone introduced in \cite{BKS}. We sketch these results in section 
\ref{surfaces}.

Throughout this paper we work over the complex numbers. 

\begin{acknowledgements}
We are grateful to Thomas Bauer and Tomasz Szemberg for helpful comments and
valuable discussions.
\end{acknowledgements}


\section{Okounkov bodies on surfaces}\label{surfaces}

In this section we first give a quick review of Okounkov's construction 
in arbitrary dimension (we refer to \cite{LM} for details), and then turn to 
additional features known in the case of surfaces.

As mentioned in the introduction, one assigns to a big divisor $D$ on a smooth 
projective $n$-dimensional variety $X$
a convex body $\Delta(D)$ in $\R^n$.
 The construction depends on the choice of
a flag on $X$, i.e., a sequence  $Y_\bullet: X=Y_0 \supset Y_1\supset\dots
\supset Y_n$ of subvarieties $Y_i$ of codimension $i$. A flag is 
\emph{admissible} if $Y_n$ is a non-singular point on each of the $Y_i$. To 
an admissible flag, one assigns a function
$$
	\valnu: H^0(X,\mathcal O_X(D))\to \Z^n,
$$
by mapping a section 
$s\in H^0(X,\mathcal O_X(D))$ to the tuple $(\nu_1(s),\dots,\nu_n(s))$ where
$\nu_1(s):=\ord_{Y_1}(s)$, $\nu_2(s)$ is given by the order of vanishing along
$Y_2$ of the section $s_1\in H^0(Y_1, \mathcal O_{Y_1}(D-\nu_1(s)Y_1))$ 
determined by $s$, and so forth up to $\nu_n(s)$. Repeating this construction 
for integral multiples of $D$, we define the Okounkov body 
$\Delta(D)=\Delta_{Y_\bullet}(D)$ to be the closed convex hull of the set
$$
	S(D) := \bigcup_{k\ge0} \set{\tfrac 1 k\valnu(s)\with s\in H^0(X, 
		\mathcal O_X(kD))}.
$$ 
Note that although the number of image vectors $(\nu_1,\dots,\nu_n)$
is equal to the dimension of $H^0(X, \mathcal O_X(kD))$ for each $k$, 
the convex body $\Delta(D)$ need not be polyhedral 
(see \cite[Section 6.3]{LM}). By \cite[Proposition 4.1]{LM}, 
numerically equivalent divisors have identical 
Okounkov bodies and for any positive integer $p$ we have the
scaling $\Delta(pD)=\tfrac 1 p \Delta(D)$, so we can assign an Okounkov 
body to big rational classes in the N\'eron-Severi vector space $N^1(X)_\R$. 
For non-rational classes this is not so straightforward. Instead, it follows 
from the existence of global Okounkov bodies (\cite[Theorem B]{LM}): There is 
a closed convex cone $\Delta(X)\subset \R^n\times N^1(X)_\R$ such that the 
fiber over any big rational class $\xi\in N^1(X)_\R$ of the map $\phi$ induced
by the second projection is equal to $\Delta(\xi)$. Consequently, the Okounkov
body of a big real class is defined as its fiber under $\phi$. Additionally, 
since the image of $\Delta(X)$ under $\phi$ is the pseudo-effective cone 
$\Eff(X)$, the construction can be extended to pseudo-effective real classes. 

From the existence of the global Okounkov body 
on $X$ many interesting properties of the volume function $\vol_X:\BigCone(X)
\to \R$ can quite easily be proved. For example, the log-concavity relation
$$
	\vol_X(D_1+D_2)^{1/n} \ge \vol_X(D_1)^{1/n}+\vol_X(D_2)^{1/n}
$$
for any two big $\R$-divisors is a consequence of the 
Brunn-Minkowski theorem: from the convexity of the global Okounkov 
body we obtain the inclusion
$$
	\Delta(D_1)+\Delta(D_2)\subset\Delta(D_1+D_2)
$$	
with the Minkowski sum on the left hand side (see \cite[Corollary 4.12]{LM}). 
\medskip

For the remainder of this section, let $X$ be a smooth projective surface with
an admissible flag
$$
	X\supset C\supset \set{p}
$$
on it.
Any pseudo-effective (rational) divisor $D$ on $X$ has a Zariski decomposition
$$
	D = P_D + N_D,
$$
where $P_D$ is nef, and $N_D$ is effective, orthogonal to $P_D$, and if it is 
not the zero-divisor, it has negative definite intersection matrix.
Define 
$$
	\mu_{C}(D):=\sup\set{t\with D-tC \mbox{ effective}}
$$
and consider the functions
$$
	\alpha,\beta:[0,\mu_{C}(D)]\to \R,
$$
with
\be
	\alpha(x) 	    & =& \ord_p(N_{D-xC}) \mbox{, and } \\
	\beta(x)          & =& \ord_p(N_{D-xC} ) + (C\cdot(P_{D-xC}) ).
\ee
Then by (\cite[Theorem 6.4]{LM}), $\alpha$ and $\beta$ are the upper 
and lower boundary functions for $\Delta(D)$, respectively. Concretely, 
$$
	\Delta(D)=\set{(x,y)\in \R^2 \with 0\le x\le \mu_{C}(D), \ \alpha(x)\le y
		\le\beta(x)}.
$$
The following proposition shows that in the situation of the theorem,
in order to determine the Okounkov body 
of a big divisor $D$ it is sufficient to know the positive part of the 
divisors
$D-tC$ for $0\le t\le \mu_C(D)$. 
\begin{proposition}\label{prop:general flag}
	If the pseudo-effective cone $\Eff(X)$ is rational polyhedral 
	and $X\supset C\supset p$ is a general admissible flag, then $C$ is 
	big and nef as a divisor, and
	$$
		\alpha(x)=0, \quad \beta(x)=C\cdot P_{D-xC}
	$$
	for all $0\le x \le\mu_C(D)$.
\end{proposition}

\begin{proof}
	If $\Eff(X)$ on $X$ is rational polyhedral,
	then in particular there are only finitely many irreducible curves $E$
	on $X$ with self-intersection $E^2\le0$. Therefore, in a general 
	flag $X\supset C\supset p$ the irreducible 
	curve $C$ has positive self-intersection, so it is 
	big and nef as a divisor. Furthermore, $p$ is a non-singular point on $C$,
	which does not lie on any curve with negative self-intersection.
	Now by definition, the negative part $N_{D-xC}$ in the Zariski 
	decomposition of $D-xC$ either is the zero-divisor, or has negative 
	definite intersection matrix. In the latter case, its support consists of 
	curves with negative self-intersection, so in either case we have 
	$\ord_p(N_{D-xC})=0$ for all $x$.
\end{proof}
\begin{example}
\emph{For any $0\le t\le1$ the class $C-tC$ is nef and effective, 
hence $P_{C-tC}=C-tC$. So by the proposition, $\Delta(C)$ is the simplex of 
height $C^2$ and length $1$. }
\end{example}

\begin{remark}\label{rmk:nefness}
	By \cite[Corollary 2.2]{L-S} the Okounkov body of a big divisor $D$
	with respect to a flag $X\supset C\supset p$ such that $C$ is not a
	component of $N_D$ is a translate by the vector $(0,\ord_p(N_D))$.
	In particular by the above proof, for a general flag on a surface with 
	rational polyhedral	pseudo-effective cone, the Okounkov bodies of any big 
	divisor and of its positive part coincide.
\end{remark}
Recall that by the main result of \cite{BKS} on a smooth
projective surface there exists a locally finite decomposition of 
$\BigCone(X)$ into locally polyhedral subcones, the so called \emph{Zariski 
chambers}, such that
\begin{itemize}
	\item the support of negative parts of divisors is constant on each 
		chamber,
	\item the volume function $\vol_X(\cdot)$ varies polynomially on the 
		chambers, and
	\item on the interior of each chamber the augmented base loci 
	$\mathbb B_+$ are constant.
\end{itemize}
The basic idea of \cite{BKS} is to consider for a big and nef divisor $P$ 
the set 
$$
	\Sigma_P:=\set{D\in \BigCone(X)\with \Neg(D)=\Null(P)},
$$
where $\Neg(D)$ denotes the support of $N_D$ and $\Null(P)$ is the
set of irreducible curves orthogonal to $P$ with respect to the intersection 
product. These sets give a decomposition of $\BigCone(X)$ 
obviously satisfying the first property in the above list, while proving 
the remaining properties as well as local finiteness still requires 
quite an effort. For an explicit description of chambers,  passing 
to closures in \cite[Proposition 1.10]{BKS} we obtain the identity
\begin{equation}\label{closure}
	\overline\Sigma_P=\mbox{convex hull}\Big(\Nef(X)\cap\Null(P)^\bot,\ 
	\Null(P)\Big),
\end{equation}
from which we deduce the following useful statement about positive parts.

\begin{proposition}\label{prop:positive parts}
	Let $P$ be a big and nef divisor on $X$ with corresponding Zariski chamber
	$\Sigma_p$. Then for all $D_1,D_2\in \overline\Sigma_P$
	we have 
	$$
    P_{D_1+D_2}=P_{D_1}+P_{D_2},
  $$
  i.e., the positive parts of the Zariski decompositions vary 
  linearly on the 
	closure of each Zariski chamber.
\end{proposition}

\begin{proof}
	Let $D_1=P_1+\sum_{i=1}^s \alpha_iN_i$ and $D_2=P_2+\sum_{i=1}^s 
	\beta_iN_i$ be representations corresponding to (\ref{closure}) with 
	$\alpha_i,\beta_i\ge0$,	$N_i\in\Null(P)$, and $P_1,P_2$ nef. Clearly, 
	$P_1+P_2$ is nef and has intersection product zero with the $N_i$. 
	Furthermore, the divisor $\sum_{i=1}^s (\alpha_i+\beta_i)N_i$ is effective
	and has negative definite intersection matrix. Thus
	$$
		D_1+D_2=(P_1+P_2)+\sum_{i=1}^s(\alpha_i+\beta_i)N_i
	$$
	is the Zariski decomposition.
\end{proof}


\section{Minkowski decomposition}\label{proof}

In this section we prove the main theorem. Fix throughout 
a general admissible flag $Y_\bullet: X\supset C\supset p$ on a smooth 
projective surface $X$. 

As stated in the introduction, 
the starting point for this investigation was the observation from 
\cite{LM} that for any two pseudo-effective divisors $D_1,D_2$ we 
have the inclusion
$$
	\Delta(D_1)+\Delta(D_2) \subset \Delta(D_1+D_2).
$$
This inclusion turns out to be strict in general. 
We refer to \cite{L-S} for examples. 

On the other hand, one observes that the Okounkov body of a pseudo-effective 
divisor $D$ with respect to $Y_\bullet$ can always be 
decomposed as the Minkowski sum of finitely many simplices 
and line segments. ($\Delta(D)$ is the area of the upper 
right quadrant bounded by the piecewise linear, concave function $\beta$.) 
The question then is: do these 
elementary \engqq{building blocks} come up as Okounkov
bodies themselves? As the theorem shows, the answer is \engqq{yes}.

Before we prove the theorem, let us consider candidates for a Minkowski basis, 
i.e., nef divisors whose Okounkov bodies are of one of the elementary types
mentioned above. 
\begin{itemize}
	\item
	For a nef divisor $D$ with $D^2=0$, for positive $t$ none of the 
	divisors $D-tC$ is effective since $C$ by Proposition 
	\ref{prop:general flag} is big and nef being the curve in a 
	general admissible flag.
	Therefore, $\mu_C(D)=0$, and $\Delta(D)$ is the vertical line segment
	of length $C\cdot D$ (see Figure \ref{fig:1}).
	
\begin{figure}[ht]
\centering
\unitlength 1mm 
\linethickness{0.4pt}
\ifx\plotpoint\undefined\newsavebox{\plotpoint}\fi 
\begin{picture}(50,39)(2,0)
\includegraphics[scale=0.8, clip, trim=5 23 0 10]{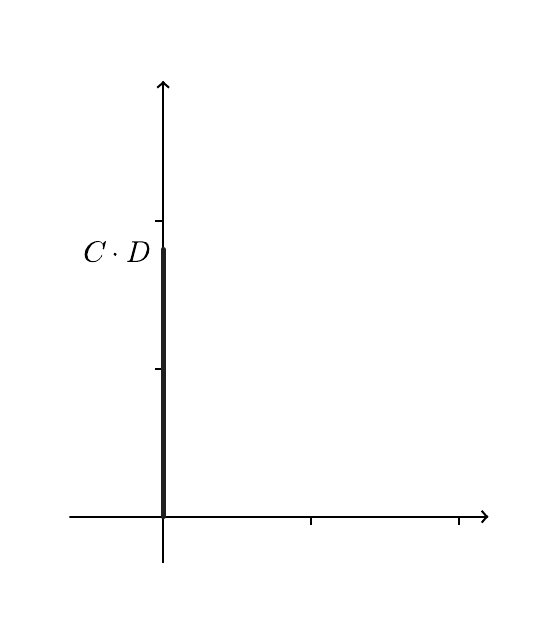}
\end{picture}
\caption{The Okounkov body $\Delta(D)$\label{fig:1}}
\end{figure}

	\item
	If for a big and nef divisor $D'$ all the classes $D'-tC$ for 
	$0<t<\mu_C(D')$ lie in the same Zariski chamber then by Proposition 
	\ref{prop:positive parts} the positive parts $P_{D'-tC}$ vary linearly 
	with $t$. Consequently,	$\Delta(D')$ is the simplex of height $C\cdot D'$ 
	and length $\mu_C(D')$ (see Figure \ref{fig:2}). 
	
	\begin{figure}[ht]
\centering
\unitlength 1mm 
\linethickness{0.4pt}
\ifx\plotpoint\undefined\newsavebox{\plotpoint}\fi 
\begin{picture}(50,45)(0,0)
\includegraphics[scale=0.7, clip, trim=5 27 0 30]{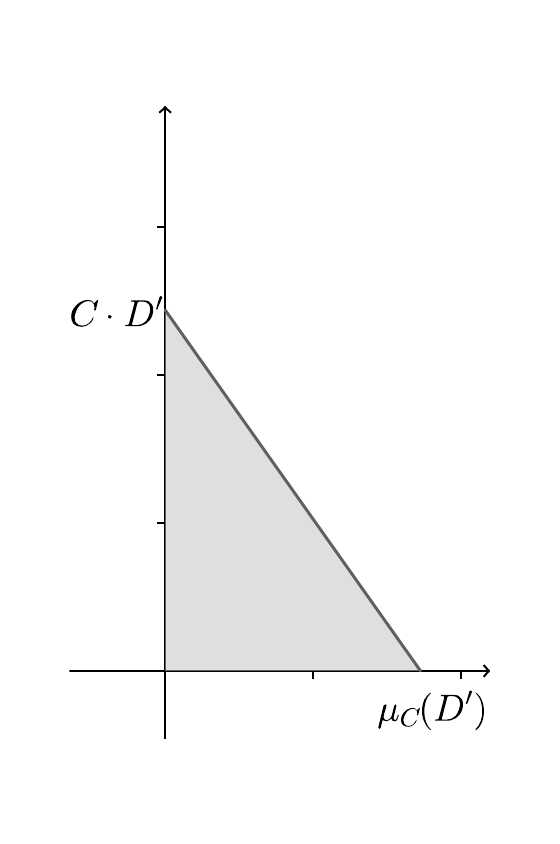}
\end{picture}
\caption{The Okounkov body $\Delta(D')$\label{fig:2}}
\end{figure}

\end{itemize}

We now turn to the proof of the theorem. It consists of two parts:
we first construct the set $\Omega$ and
then show how to find the presentation of any big and nef divisor $D$
in terms of elements of $\Omega$ which yields the Minkowski
decomposition of $D$.
\begin{remark}
	Effective representations of a nef divisor in terms of the Minkowski basis
	are not unique. It is possible that such a representation is \emph{not} a 
	Minkowski decomposition (see Example \ref{ex:not unique}). 
	This is why the second part of the proof is important as it shows how to 
	pick the right decomposition.
\end{remark}

\subsection*{Construction of a Minkowski basis}
In the Zariski chamber decomposition of the big cone $\BigCone(X)$ we
assign to each chamber an element of
$\Omega$ as follows.
Writing $\set{N_1,\dots,N_s}$ for the set of curves in the support of negative
parts of divisors in a chamber $\Sigma$, we define the \engqq{corresponding 
Minkowski basis element} $M$ as follows:  
Consider the linear subspace of $N^1(X)_\R$ spanned by $C$ together with the 
classes of the curves $N_i$. Its intersection with the subspace 
$N_1^\bot\cap\dots\cap N_s^\bot$ is a rational line, spanned by 
some integral divisor $M=dC+\sum \alpha_i N_i$. 
We will argue that $d$ and the $\alpha_i$ all have the same signs, and we
conclude that either $M$ or $-M$ is nef. 

The intersection matrix $S$ of the divisor
$\sum N_i$ is negative definite with non-negative entries outside the diagonal.
By the auxiliary result \cite[Lemma 4.1]{BKS} (see also \cite[Lemma A.1]{BF}),
the inverse matrix $S^{-1}$ has only negative entries. Therefore, and since 
$CN_i\ge 0$, the solution to the system of the equations
\begin{equation}\label{eq:system}
	S\cdot (\alpha_1,\dots,\alpha_s)^t = -d(CN_1,\dots,CN_s)^t
\end{equation}
for fixed $d$ is a vector $(\alpha_1,\dots,\alpha_s)$ whose entries have the 
same sign as $d$. Fix a positive integral solution and set 
$M=dC+\sum\alpha_iN_i$. Note that since $M$ lies in 
$N_1^\bot\cap\dots\cap N_s^\bot$ it is nef by the positivity of its 
coefficients and the nefness of $C$. Furthermore, it lies in the closure of 
$\Sigma$, or more concretely in the closure of the face 
$\overline\Sigma\cap\Nef(X)$. 
\begin{remark}
Note that in the above construction different chambers can have the same 
corresponding Minkowski basis element. For example, on the del Pezzo surface 
$X_2$ with standard basis $H,E_1,E_2$ and with a flag such that $C$ has class 
$H=\pi^\ast(\mathcal O_{\P^2}(1))$ the chambers 
$\Sigma_{H}$, $\Sigma_{2H-E_1}$, and $\Sigma_{2H-E_2}$ have $M=H$. 

Note also that the corresponding basis element to the nef cone is always $C$.
\end{remark}
We can now describe the Minkowski basis $\Omega$: it consists of the divisors 
$M_\Sigma$ constructed above together with one integral representative for 
each ray of the nef cone not contained in $\BigCone(X)$. 

Note that since $\Eff(X)$ is rational polyhedral the set $\Omega$ is finite. 
Note furthermore that the divisors in $\Omega$ have Okounkov bodies which 
cannot be decomposed as Minkowski sums, i.e., in a sense the set $\Omega$ is 
minimal:  
By construction, for all $0<t<\mu_{C}(M)=d$ the class $M_\Sigma-tC$ lies in 
the cone spanned by $\Nef(X)\cap N_1^\bot\cap\dots\cap N_s^\bot$ and the 
$N_1,\dots,N_s$, i.e., in the closure of the Zariski chamber $\Sigma$. 
Therefore, the positive part of $M_\Sigma-tC$ varies linearly, so 
$\Delta(M_\Sigma)$ is the simplex of height $C\cdot M$ and length $d$, whereas
the other basis elements $D_i$ lie in the boundary of $\Eff(X)$, so 
$\mu_{C}(D_i)=0$ which means that the corresponding Okounkov body is the 
vertical line segment of length $C\cdot D_i$.

\subsection*{Algorithmic construction of Minkowski decompositions}
To complete the proof we now describe how to find the Minkowski decomposition 
of a given nef divisor $D$.

If $D$ is not big, then $D^2=0$ and $\Omega$ contains some positive multiple 
$D'=\beta D$. Thus
$$
	\Delta(D)=\frac1\beta\Delta(D'),
$$
and we are done.

Otherwise, consider the Zariski chamber $\Sigma$ corresponding to the big and 
nef divisor $D$. Let $M$ be the corresponding Minkowski basis element and set 
$$
	\tau:=\sup\set{t\with D-tM \mbox{ nef}}.
$$ 
Since nefness is defined by finitely many linear conditions, $\tau$ is 
rational. The nef $\Q$-divisor $D' := D-\tau M$ lies on the boundary of the 
face $\Nef(X)\cap\Null(D)$.
If $D'=0$, we are done. 

Otherwise, we claim that
\begin{equation}\label{eq:separation}
	\Delta(D)= \tau\Delta(M) + \Delta(D'),
\end{equation}
so $\Delta(D)$ decomposes into the elementary part $\tau\Delta(M)$ and
the Okounkov body of the divisor $D'$.

For the proof we first note that by construction of the Minkowski basis, $M$ 
lies, like $D'$, on the boundary of the Zariski chamber $\Sigma$. Furthermore,
as we have seen above, the divisors $M-tC$ lie in the closure of the chamber 
$\Sigma$ for $0<t<\mu_C(M)$. 
Thus by Proposition \ref{prop:positive parts} we have 
$$
	P_{D-xC}=P_{D'}+P_{\tau M-xC} = D' + P_{\tau M-xC}
$$
for $0\le x \le \mu_{C}(\tau M)$. 

For the remaining 
$\mu_{C}(\tau M)\le x \le \mu_{C}(D)$ let $\tilde M$ denote the divisor 
$\tau M-\mu_{C}(\tau M)C$ (which 
is just $\tau(\sum \alpha_i N_i)$ in the above notation). We claim that for 
any $t>0$ we have the inclusions
$$
	\supp(\tilde M)\subset\Null(D)\subset\Null(D')=\mathbb B_+(D')\subset 
	\mathbb B_+(D'-tC)=\Null(P_{D'-tC}).
$$
The two equalities are given by \cite[Example 1.10]{ELMNP} and 
\cite[Example 1.11]{ELMNP} respectively. The first inclusion is clear since 
the $N_i$ are contained in $\Null(D)$. The second one follows from the fact 
that $D'$ is contained in the boundary of the face of the nef cone containing 
$D$, while the last inclusion is a direct consequence of the fact that 
subtracting a nef divisor can only augment the base locus. 

Note that in general for a big divisor $E$ with Zariski decomposition 
$E=P_E+N_E$ and an effective divisor $F$ with support contained in 
$\Null(P_E)$ the decomposition
$$
	E+F= P_E + (N_E + F)	
$$
is the Zariski decomposition: $P_E$ is nef, has trivial intersection with all 
components of $(N_E +F)$, and the latter divisor has negative definite 
intersection matrix. In other words, adding an effective divisor $F$ with 
support contained in $\Null(P_D)$ does not alter the positive part.

Taking in the above consideration $E$ and $F$ to be $D-xC$ and $\tilde M$ 
respectively, we obtain the identity
$$
	P_{D-xC}=P_{D'-(x-\mu_{C}(\tau M))C} 
$$	
for $\mu_{C}(\tau M)\le x \le \mu_{C}(D)$.	
Putting the two decompositions of positive parts together, we get
$$
	\beta_D(x)=\begin{cases}
				\beta_{\tau M}(x) + C\cdot D'	& 0\le x \le \mu_{C}(\tau M)\\
				\beta_{D'}(x-\mu_{C}(\tau M)) & \mu_{C}(\tau M)\le x \le 
					\mu_{C}(D),
									\end{cases}
$$
which amounts to the claimed identity (\ref{eq:separation}).

Repeat the above procedure with the divisor $D'$. This is possible because 
if $D'$ is big and nef, it defines a Zariski chamber $\Sigma$ with $M_{\Sigma}
\neq M$, which can be seen as follows: if it were not the case, we would have 
$\Null(D')\subset\Null(M)$, but then it follows from $D=M+D'$ that 
$\Null(D')\subset\Null(D)$, which is impossible. The algorithm terminates 
after at most $\rho$ steps, since in every step the dimension of the face of 
the nef cone in which $D$ lies decreases. Eventually, we end up with either 
$0$ or a divisor spanning an extremal ray of the nef cone. Such a divisor has 
a multiple in $\Omega$, and we are done. \hfill$\Box$\medskip

Note that in order to determine the Minkowski decomposition of a given 
divisor $D$ it is not necessary to know the whole Minkowski basis of $X$.
Instead in every step the necessary basis element can be found based on 
knowledge of the intersection matrix of $\Null(D)$ alone. In fact, the
algorithm can be implemented for automated computation, provided the 
intersection matrix of $C$ together with the negative curves on $X$ is known.


\section{Del Pezzo surfaces}
On a del Pezzo surface $X$ the pseudo-effective cone is rational polyhedral by
the cone theorem. Concretely, it is spanned by rational curves of 
self-intersection $-1$. The surface $X$ is either $\P^2$, its blow-up $X_r$ in
up to 8 general points, or $\PtP$. A complete list of the $(-1)$-curves on the
$X_r$ is well known \cite[8 Chapt IV]{M} (cf. \cite[Theorem 3.1]{BFN} for an 
elementary proof): they are the exceptional curves $E_1,\dots,E_r$ together 
with the strict transforms of
\begin{itemize}
	\item
		lines through two of the $p_i$,
	\item
		irreducible conics through five of the $p_i$, if $r\ge5$,
	\item
		irreducible cubics through six of the $p_i$ with a double point in 
		one of them, if $r\ge7$,
	\item
		irreducible quartics through the eight points $p_i$ with 
		a double point in three of them, if $r\ge8$,
	\item
		irreducible quintics through the eight points $p_i$ with
		a double point in six of them, if $r\ge8$,
	\item
		irreducible sextics through the eight points $p_i$ with
		a double point in seven of them, and a triple point in one of them, 
		if $r\ge8$.
\end{itemize}
A general flag on $X_r$ consists of an irreducible curve $C$ with a general 
point $p$ on it where $C$ is the strict transform of an irreducible member of 
the class $\mathcal O_{\P^2}(k)$ for some $k>0$. We consider the case $k=1$ 
(the others work analogously) and write as usual $H$ for the class of $C$.
Let us construct a Minkowski basis for $X_r$. Starting with any chamber 
$\Sigma$, we consider $\Neg(\Sigma)=\set{N_1,\dots,N_s}$, the support of 
negative parts of the divisors in $\Sigma$. Its intersection matrix, being 
negative definite with diagonal entries $-1$, can have only zero entries 
outside the diagonal. In particular, we can immediately read off the basis 
element $M(\Sigma)$ from the system of equations (\ref{eq:system}):
Setting $d=1$, we obtain $\alpha_i={N_i\cdot H}$ for all $i$, hence we have
$$
	M(\Sigma)=H+\sum_{i=1}^s ({N_i\cdot H}) N_i.
$$
 Let us determine the Okounkov bodies of this Minkowski basis element.
 It is clear that $\mu_H(M(\Sigma))=1$, since $\sum_{i=1}^s ({N_i\cdot H}) N_i$
 lies on the boundary of $\Eff(X_r)$. On the other hand, setting
 $$
         \lambda:=H\cdot(H+\sum_{i=1}^s ({N_i\cdot H}) N_i),
$$ 
by the argumentation in the proof of the theorem, $\Delta(M(\Sigma))$ is the 
simplex of height $\lambda$ and length $1$, which we denote by 
$\Delta(\lambda,1)$. The remaining elements of $\Omega$ are curves $E$ with 
self-intersection $E^2=0$. As we have seen above, their Okounkov body is the 
vertical line segment of length $H\cdot E$. The following statement thus is a 
direct consequence of the theorem.
\begin{proposition}
	On a del Pezzo surface $X_r$, for any big divisor $D\subset X_r$ the 
	function $\beta(x)$ bounding the Okounkov body is piecewise linear with 
	integer slope on each linear piece.
\end{proposition}

For a concrete calculation, consider the del Pezzo surface $X_6$. Up to 
permutation of the $E_i$, we have the possible supports for Zariski chambers 
with corresponding basis elements displayed in Table \ref{chambers} in the 
standard basis $H,E_1,\dots,E_r$, with $L_{i,j},C_1,C_2$ denoting the $(-1)$-curves 
coming from lines and conics, respectively.

\renewcommand{\arraystretch}{1.5}
\newcolumntype{C}[1]{>{\centering\arraybackslash}p{#1}}
\begin{table}[ht]\centering\footnotesize
      \begin{tabular}[t]{c|C{60mm}}\tline
         $\Neg(\Sigma) $ &  $M(\Sigma)$   \\\tline
         $E_1,\dots,E_s$ 		&  $H$          	  \\
         
         $L_{1,2},\dots,L_{1,1+s},E_{s+1},\dots,E_{s+t}$ &   
			$(s+1)H-sE_1-E_2-\dots-E_s$   \\
         
         $L_{1,2},L_{1,3},L_{2,3},E_{4},\dots,E_{4+t}$ &   
			$4H-2E_1-2E_2-2E_3$   \\
         
         $C_1,L_{2,3},\dots,L_{2,2+s},E_{s+1},\dots,E_{s+t}$ &  
         $(5+s)H-(2+s)E_2-3E_3-\dots-3E_s$ $-2E_{s+1}-\dots-2E_6$\\
         
         $C_1,L_{2,3},L_{2,4},L_{3,4},E_{1}$ &  
         $8H-4E_2-4E_3-4E_4$ $-2E_5-\dots-2E_6$         	\\
         
         {$C_1,C_2, L_{3,4},\dots,L_{3,3+s}$ }&  
			{$(9+s)H-2E_1-2E_2-{(s+4)}E_3$  
			$-5E_4-\dots-5E_{3+s}-4E_{4+s}-\ldots-4E_r$ }\\
         
         {$C_1,C_2, L_{3,4},L_{3,5},L_{4,5}$ }&  
			{$12H-2E_1-2E_2-6E_3-6E_4-6E_5-4E_6$ }    
			\\\tline
      \end{tabular}
      \caption{Zariski chambers and corresponding Minkowski basis elements 
	  on $X_6$\label{chambers}}
\end{table}

The additional Minkowski basis elements (corresponding to non-big nef classes) 
are the strict transforms of 
\begin{itemize}
	\item lines through one of the $p_i$,
	\item irreducible conics through four of the $p_i$.
\end{itemize}
We thus get the following elementary bodies as building blocks for the
Okounkov body of any big divisor on $X_6$:
$$
	\Delta(1,1),\ldots,\ \Delta(12,1),\ \Delta(1,0),\ \Delta(2,0).
$$

\begin{example}\label{ex:not unique}
	\emph
	{Consider the divisor $D=7H-2E_1-E_2-3E_3-2E_4-2E_5$ on $X_6$.
	\begin{itemize}
	\item
		For $D_1=D=7H-2E_1-E_2-3E_3-2E_4-2E_5$, we find $\Null(D)=\set{E_6}$, 
		so $M(D)=H$. With $\tau=2$ we get
		$D_2=5H-2E_1-E_2-3E_3-2E_4-2E_5$.
	\item
		Now, $\Null(D_2)=\set{C_6,L_{1,3},L_{3,4},L_{3,5},E_6}$, so
		 $M(D)=8H-3E_1-2E_2-5E_3-3E_4-3E_5$. With
		 $\tau=\tfrac12$ we get
		$D_3=H-\tfrac12E_1-\tfrac12E_3-\tfrac12E_4-\tfrac12E_5$.
	\item
		then $D_3^2=0$, so we are done.
	\end{itemize}
	Cosequently, the Okounkov body of $D$ is given as the Minkowski sum
	$$
		\Delta(D)=\Delta(2,2)+\tfrac12\Delta(8,1)+\Delta(1,0)
	$$			
	depicted in Figure \ref{fig:3}}.
\begin{figure}[ht]
\centering
\unitlength 1mm 
\linethickness{0.4pt}
\ifx\plotpoint\undefined\newsavebox{\plotpoint}\fi 
\begin{picture}(50,88)(0,0)
\includegraphics[scale=0.7, clip, trim=20 30 20 38]{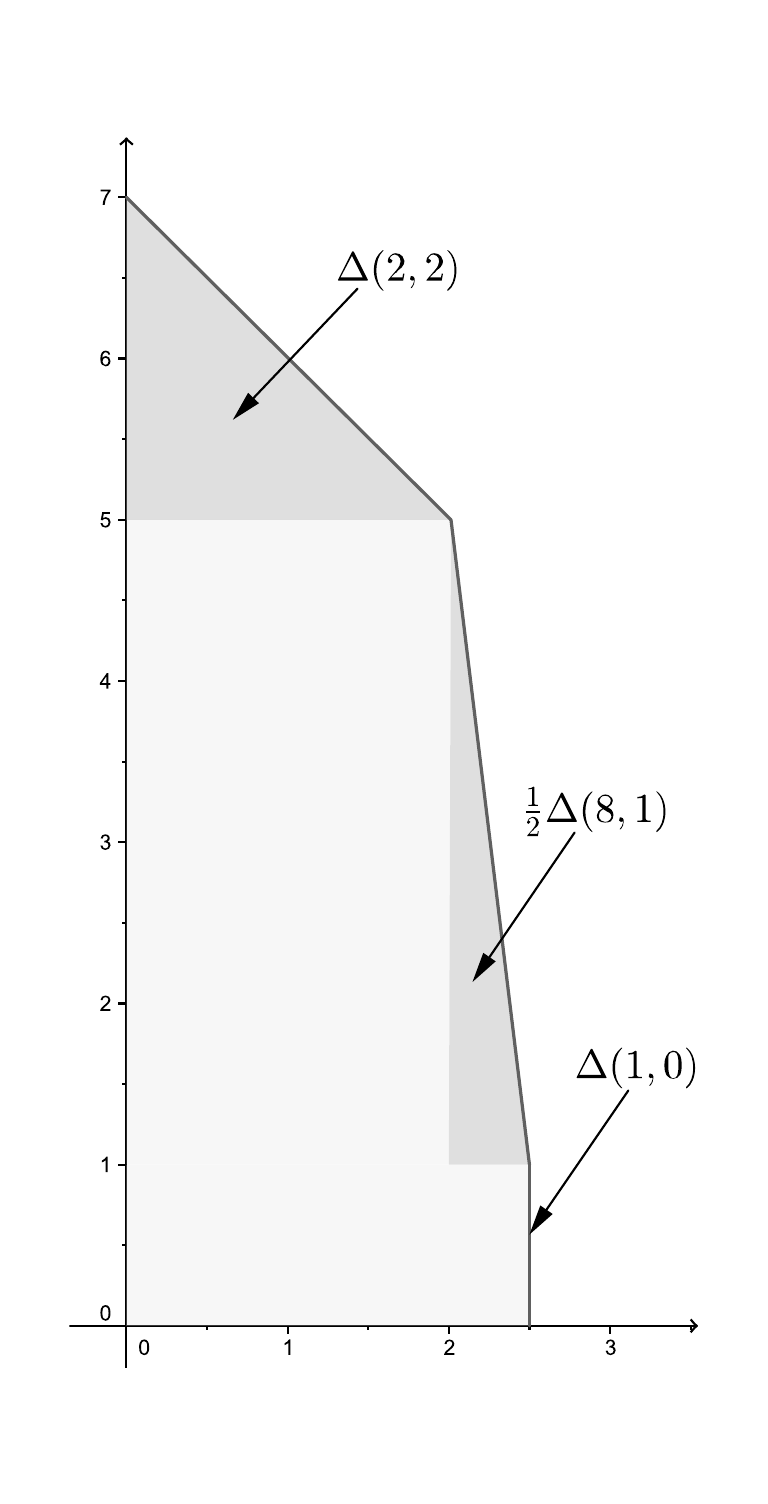}
\end{picture}
\caption{The Okounkov body $\Delta(D)$ as a Minkowski sum\label{fig:3}}
\end{figure}

\emph{
	Note on the other hand that we have the identity
	$$
		D=(3H-2E_1-E_2-E_3) + (4H-2E_3-2E_4-2E_5)
	$$
	and both summands are Minkowski basis elements. Clearly, this 
	representation cannot be a Minkowski decomposition (see Figures 
	\ref{fig:4} and \ref{fig:5}). 
	}
	\begin{figure}
	\begin{minipage}[hbt]{7cm}
	\centering
		\includegraphics[scale=4, clip, trim=20 45 20 45]{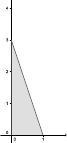}
		\caption{$\Delta(3H-2E_1-E_2-E_3)$\label{fig:4}}
\end{minipage}
\hfill
\begin{minipage}[hbt]{7.5cm}
	\centering
		\includegraphics[scale=4, clip, trim=20 45 20 45]{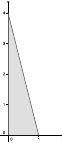}
		\caption{$\Delta(4H-2E_3-2E_4-2E_5)$\label{fig:5}}
\end{minipage}
\end{figure}

\end{example}


\section{Non-del-Pezzo examples}\label{sec:examples}
\begin{enumerate}
\item
For a simple non-del-Pezzo example, let $\pi:X\to\P^2$ be the blow-up of 3 
points on a line with exceptional divisors $E_1,E_2,E_3$. Choose $C$ general 
in the class $H:=\pi^\ast(\mathcal O_{\P^2}(1))$ and $p\in C$ a general point.
This gives a flag as above. The pseudo-effective cone is spanned by the 
exceptional divisors together with the class $D:=H-E_1-E_2-E_3$ of the strict 
transform of the line joining the blown up points. We have $12$ Zariski 
chambers: the nef chamber,  the $7$ chambers belonging to principal 
submatrices of the intersection matrix of $E_1+E_2+E_3$, the one corresponding
to $D$,and three chambers with support $D$ together with one of the 
exceptional divisors. The corresponding Minkowski basis element is $H$ for the
first $8$ chambers, $3H-E_1-E_2-E_3$ for the 9th, and $2H-E_i-E_j$ for last 
three. The remaining elements of $\Omega$ are $H-E_1, H-E_2, H-E_3$. Let's 
calculate the decomposition for the arbitrarily chosen divisor 
$P=15H-3E_1-3E_2-E_3$. 
\begin{itemize}
\item The divisor $P$ is ample, so $M=H$; with $\tau=8$ we get 
	$P_1=7H-3E_1-3E_2-E_3$.
\item Now, $\Null(P_1)=D$, so $M_\Sigma=3H-E_1-E_2-E_3$; with $\tau=1$ we get
	$P_2=4H-2E_1-2E_2$.
\item In the next step, $\Null(P_2)=\set{D,E_3}$, so $M_\Sigma=2H-E_1-E_2$; 
	with $\tau=2$, we get $P_3=0$, and we are done.
\end{itemize}
Thus we get the decomposition
$$
	P=8\cdot H + (3H-E_1-E_2-E_3) + 2\cdot(2H-E_1-E_2) 
$$
with corresponding Minkowski decomposition of the Okounkov body
\be
	\Delta(P) &=& 8\Delta(H) + \Delta(3H-E_1-E_2-E_3) + 2\Delta(2H-E_1-E_2)\\
						&=& \Delta(8,8)+\Delta(3,1)+\Delta(4,2).\\
\ee

\item (K3-surface)\\
For an example of a surface which is not a blow-up of $\P^2$ let us consider a
K3-surface. As Kov\'acs proves in \cite{Kov}, for any $1\le \rho\le 19$ there
exists a K3-surface $X$ with Picard number $\rho$ whose pseudo-effective cone 
is rational polyhedral, spanned by the classes of finitely many rational 
$(-2)$-curves. We consider a certain K3-surface of this type: It was proved in
\cite[Proposition 3.3]{BF} that there exists a K3-surface $X$ with Picard 
number $3$ such that the pseudo-effective cone is spanned by three 
$(-2)$-curves $L_1,L_2,D$ forming a hyperplane section $L_1+L_2+D$ such that
$L_1$ and $L_2$ are lines and $D$ is an irreducible conic. The hyperplane 
section $L_1+L_2+D$ has intersection matrix
$$
	\begin{pmatrix} -2 & 1 & 2\\
			1& -2 & 2\\
			2&2&-2

\end{pmatrix}.
$$
Therefore, the Zariski chamber decomposition consists of five chambers, namely
the nef chamber, one chamber corresponding to each of the $(-2)$-curves 
$D,L_1,L_2$, and one chamber with support $L_1+L_2$. Pick $C$ to be an 
irreducible curve with class $L_1+L_2+D$, i.e., a general hyperplane section, 
and $p$ to be a point in $C$ not on $L_1,L_2$, and $D$. Then the Minkowski 
basis elements corresponding to the above list of chambers are $C$, 
$3L_1+2L_2+2D$, $2L_1+3L_2+2D$, $L_1+L_2+2D$, and $2L_1+2L_2+D$. In addition, 
the Minkowski basis $\Omega$ contains the curves $L_1+D$ and $L_2+D$ of 
self-intersection zero. Thus, by the theorem, the building blocks of Okounkov 
bodies of nef divisors on $X$ are
$$
	\Delta(4,1),\Delta(9,2),\Delta(6,1),\Delta(3,0).
$$
In particular, in contrast to the del Pezzo case, the slope of a linear piece 
of the bounding function $\beta$ need not be integral for K3-surfaces.

\end{enumerate}


	
	\bigskip
	\small
	Patrycja \L{}uszcz-\'Swidecka
	Institute of Mathematics,
	Jagiellonian University, 
	ul.~\L ojasiewicza 6, 30-438 Krak\'{o}w, Poland.

	\nopagebreak
	\textit{E-mail address:} \texttt{paczic@gmail.com}

	\bigskip
	David Schmitz,
	Fach\-be\-reich Ma\-the\-ma\-tik und In\-for\-ma\-tik,
	Philipps-Uni\-ver\-si\-t\"at Mar\-burg,
	Hans-Meer\-wein-Stra{\ss}e,
	D-35032~Mar\-burg, Germany.

	\nopagebreak
	\textit{E-mail address:} \texttt{schmitzd@mathematik.uni-marburg.de}

\end{document}